\documentclass{amsart}
\usepackage{graphicx,amsmath,amsthm,amssymb,enumerate,overpic}

\usepackage[all,cmtip]{xy}

\usepackage[margin=1.2in]{geometry}
\usepackage[dvipsnames]{xcolor} 
\usepackage{hyperref}
\hypersetup{colorlinks=true,
allcolors=NavyBlue}

\newcommand{\R}{\mathbb{R}}
\newcommand{\Z}{\mathbb{Z}}
\newcommand{\N}{\mathbb{N}}
\newcommand{\Pf}{\mathcal{P}}

\newcommand{\C}{\mathcal{C}}
\newcommand{\G}{\mathcal{G}}
\newcommand{\st}{\;|\;}
\newcommand{\Cdagger}{\C^\dagger}
\newcommand{\NC}{\mathcal{NC}}
\newcommand{\NCdagger}{\NC^\dagger}
\DeclareMathOperator{\Homeo}{Homeo}
\DeclareMathOperator{\mcg}{MCG}
\newcommand{\tl}[1]{\operatorname{tl}(#1)}
\newcommand{\ssm}{\smallsetminus}

\newcommand{\dist}{d}
\newcommand{\distdagger}{d^\dagger}
\newcommand{\sat}[1]{{#1}^{\textrm{sat}}}
\newcommand{\diam}{\operatorname{diam}}
\newcommand{\Span}[1]{\left\langle #1\right\rangle}
\newcommand{\Csurv}{\C^{\text{surv}}}
\newcommand{\AC}{\mathcal{AC}}

\newtheorem{thm}{Theorem}[section]
\newtheorem{lemma}[thm]{Lemma}
\newtheorem{prop}[thm]{Proposition}

\newtheorem{cor}[thm]{Corollary}

\theoremstyle{definition}
\newtheorem{rmk}[thm]{Remark}

\title[Approximating stable translation lengths on fine curve graphs]{Approximating stable translation lengths\\ on fine curve graphs}
\author{Federica Fanoni}
\email{federica.fanoni@u-pec.fr}
\address{CNRS, Univ Paris Est Creteil, Univ Gustave Eiffel, LAMA UMR8050, F-94010 Creteil, France}
\author{Sebastian Hensel}
\email{hensel@math.lmu.de}
\address{Mathematisches Institut der Universit\"at M\"unchen, Theresienstr. 39, 80333 M\"unchen, Germany}
\author{Frédéric Le Roux}
\email{frederic.le-roux@imj-prg.fr}
\address{IMJ-PRG, Sorbonne Université, 4 place Jussieu, 75252 Paris Cedex 05, France}
\date{\today}

\begin{document}

\begin{abstract}
We study the stable translation length of homeomorphisms of a surface acting on the fine nonseparating curve graph and compare it to the stable translation lengths of its finite approximations --- mapping classes relative to a finite invariant set --- acting on the nonseparating curve graph. We prove that the stable translation length of a homeomorphism with a dense set of periodic points is the supremum of the stable translation lengths of its approximations, and that the stable translation length is preserved under cell-like extensions. We deduce that homotopically trivial homeomorphisms of the torus have stable translation length which is the supremum of the stable translation lengths of their finite approximations. We show that the supremum is not always a maximum, by proving that the stable translation length of a mapping class acting on the nonseparating curve graph is rational.
\end{abstract}

\maketitle

\section{Introduction}

Given a surface \(S\), a classically studied object is the group \(\Homeo(S)\) of its homeomorphisms. Since a few years, this group and its elements have been studied through the lenses of geometric group theory, by looking at the action of \(\Homeo(S)\) on the \emph{fine curve graph} \(\Cdagger(S)\) (first introduced in \cite{bhw_quasi}). This viewpoint has already been very useful to understand algebraic properties of this group and to link coarse geometric aspects to dynamical features of homeomorphisms (see e.g.\ \cite{bhw_quasi, bhmmw_rotation, gm_parabolic, gm_hyperbolic, einabadi_torus, bhw_towards}).

A key tool from \cite{bhw_quasi} is a local approximation result (Lemma \ref{lem:local-approximation} below), which --- roughly speaking --- says that distances in fine (nonseparating) curve graphs can be approximated via much better known graphs, the (nonseparating) curve graphs of punctured surfaces. In this article, we are interested in the stable translation length of homeomorphisms acting on fine nonseparating curve graphs and in whether they can be approximated via stable translation lengths of associated mapping classes acting on nonseparating curve graphs of punctured surfaces.

In what follows, \(S\) is a closed orientable surface of genus at least one and \(\NCdagger(S)\) denotes the nonseparating fine curve graph of \(S\) (see Section~\ref{sec:prelim} for the definitions). If \(P\) is a finite set of points of \(S\), \(\NC(S\ssm P)\) denotes the nonseparating curve graph of the surface \(S\ssm P\). Given a homeomorphism \(f\) of a surface \(S\), we denote by \(\tl{f}\) its stable translation length on \(\NCdagger(S)\). Let \(\Pf(f)\) be the set
\[\Pf(f):=\{P\subset S\st f(P)=P, |P|<\infty\}.\]

For \(P\in\Pf(f)\), we denote by \([f]_P\) the mapping class \([f|_{S\ssm P}\in\mcg(S\ssm P)\) and by \(\tl{[f]_{P}}\) its stable translation length on \(\NC(S\ssm P)\). We call \([f]_P\) a \emph{finite approximation} of \(f\).

Our first result shows that if a homeomorphism has many periodic points, its stable translation length is the supremum of the stable translation lengths of its finite approximations.

\begin{thm}\label{thm:dense}
Let \(S\) be a closed orientable surface of genus at least one and \(f\in\Homeo(S)\). If the set of periodic points of \(f\) is dense in \(S\), then 
    \[\tl{f}=\sup_{P\in\Pf(f)}\tl{[f]_{P}}.\]
\end{thm}

We then show that stable translation lengths of homeomorphisms which are semi-conjugated via a cell-like map are equal:

\begin{thm}    \label{thm:semi-conjugacy}
Let \(S,\bar{S}\) be closed orientable surfaces of genus at least one and \(f\in\Homeo(S), \bar{f}\in\Homeo(\bar{S})\). Suppose \(q:S\to \bar{S}\) is a cell-like semi-conjugacy between \(f\) and \(\bar{f}\). Then:
    \begin{enumerate}
        \item \(\tl{f}=\tl{\bar{f}}\);
        \item for every \(\bar{f}\)-invariant finite subset \(\bar{P}\subset \bar{S}\) and every \(f\)-invariant  finite subset \(Q\subset S\) such that \(q(Q)=\bar{P}\), we have
    \[\tl{[f]_{Q}}=\tl{[\bar{f}]_{\bar{P}}}.\]
    \end{enumerate}
\end{thm}

As a consequence of the two theorems, using work of García-Sassi and Tal \cite{gst_fully}, we can show that stable translation lengths of homotopically trivial homeomorphisms of the torus are approximated by the stable translation lengths of their finite approximations:

\begin{cor}\label{cor:torus-case}
For any \(f\in\Homeo_0(T^2)\), we have
\[\tl{f}=\sup_{P\in\Pf(f)}\tl{[f]_{P}}.\]
\end{cor}

We also know that in general the supremum in the corollary is not realized: indeed, since the translation length is continuous, there are homotopically trivial homeomorphisms of the torus with irrational asymptotic translation length, while the following result shows that the asymptotic translation lengths of the mapping classes are always rational.

\begin{thm}\label{thm:rational}
Let \(S\) be a (not necessarily closed) surface of genus at least one. For any \(\varphi\in\mcg(S)\) acting hyperbolically on \(\NC(S)\), there is \(m>0\) such that \(\varphi^m\) preserves a geodesic. In particular, the stable translation length of \(\varphi\) is rational.
\end{thm}

Rationality of stable translation lengths holds also for mapping classes acting on the curve graph, by work of Bowditch \cite{bowditch_tight}, or on the pants graph, as shown by Irmer \cite{irmer_stable}. On the other hand, note that the inclusion \(\NC(S)\hookrightarrow \C(S)\) is highly distorted, so we cannot use the result of Bowditch to deduce rationality of stable translation lengths on the nonseparating curve graph.

Finally, we show that for a homeomorphisms which is pseudo-Anosov relative to a finite set of points, the stable translation length is the same as that of the mapping class relative to the angle-\(\pi\) singularities:

\begin{prop}\label{prop:pA-is-rational}
Let \(f\in\Homeo(S)\) be a pseudo-Anosov homeomorphism relative to a finite set of points. Let \(P_\pi\) be the set of angle-\(\pi\) singularities. Then
\[\tl{f}=\tl{[f]_{P_\pi}}.\]
In particular, \(\tl{f}\in \mathbb{Q}\).
\end{prop}

As a consequence, the supremum in Corollary \ref{cor:torus-case} is sometimes realized.

\section*{Acknowledgements}

The first author was partially supported by the ANR grants MAGIC (ANR-23-TERC-0007) and GALS (ANR-23-CE40-0001). The first two authors are grateful to the SPP Main Conference (Leipzig, March 31 to April 4, 2025), where part of the work was carried out.

\section{Preliminaries}\label{sec:prelim}

Throughout this article, surfaces are assumed to be connected, orientable and of finite type, and subsurfaces are incompressible, nonperipheral and connected. Unless otherwise stated, surfaces will be assumed to have no boundary, one notable exception being subsurfaces of a given surface. The \emph{complexity} \(\xi(S)\) of a (sub)surface \(S\) is \(3g+n+b\), where \(g\) is the genus of \(S\), \(n\) its number of punctures and \(b\) its number of boundary components.

Curves are assumed to be simple and closed. Given a curve \(\alpha\) on a surface \(S\), we denote by \([\alpha]\) its free homotopy class. Unless otherwise stated, curves will also be assumed to be \emph{essential} (i.e.\ not bounding a disk with at most one puncture) and \emph{nonperipheral} (i.e.\ not homotopic to a boundary component).

Two curves \(\alpha\) and \(\beta\) are said to be \emph{in minimal position} if \(|\alpha\cap\beta|=\min\{|\alpha'\cap\beta'|\st \alpha'\in [\alpha], \beta'\in [\beta]\}.\)

Given two (distinct) homotopy classes of curves \(a\) and \(b\), the subsurface \emph{filled by \(a\) and \(b\)}, denoted \(\Span{a,b}\), is the homotopy class of a subsurface obtained as follows: choose representatives \(\alpha\in a\) and \(\beta\in b\) in minimal position, take a regular neighborhood \(N\) and fill in every disk with at most one puncture bounded by a component of \(\partial N\). Note that if $c,d$ are curves in $\Span{a,b}$ then $\Span{c,d} \subset \Span{a,b}$ up to homotopy.

An \emph{arc} on a surface with boundary is the image of \([0,1]\) under an embedding, sending the endpoints of the interval to the boundary of the surface. Arcs are assumed to be \emph{essential}, i.e.\ not homotopic to an arc in the boundary.

\subsection*{Gromov hyperbolic metric spaces}
A geodesic metric space \(X\) is \emph{\(\delta\)-hyperbolic} if all its geodesic triangles are \(\delta\)-thin (i.e.\ each side is contained in the \(\delta\)-neighborhood of the other two sides). Given such a space and an isometry \(f\), the \emph{stable translation length} \(\tl{f}_X\) is given by
\[\tl{f}_X:=\lim_{n\to\infty}\frac{\dist(x,f^n(x))}{n},\]
where \(x\) is a point in \(X\). It is well known that the limit exists and does not depend on the point \(x\). The isometry is said to be \emph{hyperbolic} if \(\tl{f}_X>0\). If \(\tl{f}_X=0\), it is \emph{elliptic} if the orbits are bounded and \emph{parabolic} otherwise.

A sequence \((x_n)_{n\in\Z}\subset X\) is a \emph{\(K\)-quasigeodesic} if for every \(n,m\in\Z\) we have
\[\frac{1}{K}|n-m|-K\leq \dist(x_n,x_m)\leq K|n-m|+K.\]
It is a \emph{\(N\)-local \(K\)-quasigeodesic} if the above inequalities hold whenever \(|n-m|<N\).

We will use two well-known results about \(\delta\)-hyperbolic
spaces. The first one allows to find quasi-axes of a definite quality
(see e.g. \cite[Lemma 3.6]{hlr_rotation} for this version):
\begin{lemma}\label{lem:quasigeods}
Let \(X\) be a \(\delta\)-hyperbolic space.
\begin{enumerate}
    \item There are \(K=K(\delta), L=L(\delta)\) such that if \(f\) is an isometry of \(X\) with \(\tl{f}_X\geq L\) and \(x\in X\) minimizes \(\dist(x,f(x))\), then \[\bigcup_{n\in\Z}f^n([x,f(x)])\]
    is a \(K\)-quasigeodesic, for any choice of geodesic segment \([x,f(x)]\).
    \item Let \(K'>0\) and let \(f\) be an isometry of \(X\). If \((f^n(x)){n\in\Z}\) is a \(K'\)-quasigeodesic, then for every \(n,k\geq 1\)
    \[d(x,f^{nk}(x))\geq n(\dist(x,f^k(x))-2M),\]
    where \(M=M(K',\delta)\) is the Morse constant for \(K'\)-quasigeodesics. In particular, for every \(k\geq 1\)
    \[\tl{f}_X\geq \frac{d(x,f^k(x))}{k}-\frac{2M}{k}.\]
\end{enumerate}
\end{lemma}

The second result is a well-known local-to-global property (see e.g.\ \cite[Théorème 1.4, Chapitre 3]{cdp_geometrie}):

\begin{lemma}\label{lem:local-to-global}
Let \(X\) be a \(\delta\)-hyperbolic space. For every \(K>0\) there are \(N=N(K),K'=K'(K)>0\) such that if \(\gamma\) is an \(N\)-local \(K\)-quasigeodesic, then \(\gamma\) is a \(K'\)-quasigeodesic.
\end{lemma}

\subsection*{(Fine) curve graphs}

Suppose that \(S\) is a surface of genus at least one. The \emph{fine curve graph} \(\Cdagger(S)\) of a surface \(S\) is the graph whose vertices correspond to curves and where two vertices are joined by an edge if the corresponding curves:
\begin{itemize}
\item are disjoint, if the genus of \(S\) is at least two,
\item intersect at most once, if the genus is one.
\end{itemize}
The \emph{curve graph} \(\C(S)\) is the graph whose vertices correspond to homotopy classes of curves and two vertices are joined by an edge if there are curves in the classes which:
\begin{itemize}
\item are disjoint, if the genus of \(S\) is at least two,
\item intersect at most once, if the genus is one.
\end{itemize}
The subgraph of \(\Cdagger(S)\) whose vertices are nonseparating curves is the \emph{nonseparating fine curve graph} \(\NCdagger(S)\), and the subgraph of \(\C(S)\) whose vertices are classes of nonseparating curves is the \emph{nonseparating curve graph} \(\NC(S)\). Note that \(\Cdagger(S)\), \(\NCdagger(S)\), \(\C(S)\) and \(\NC(S)\) are \(\delta\)-hyperbolic, and the constant \(\delta\) can be chosen independently on the surface (see \cite{bhw_quasi, mm_geometryI,aougab_uniform,bowditch_uniform,crs_uniform,hpw_unicorns}). 
The homeomorphism group \(\Homeo(S)\) acts on \(\Cdagger(S)\) by isometries, and the \emph{mapping class group} \(\mcg(S)\) (the group of orientation-preserving homeomorphisms, up to homotopy) acts on \(\C(S)\) by isometries.

Recall from the introduction that given a homeomorphism \(f\) of a surface \(S\), we denote by \(\tl{f}\) its stable translation length on \(\NCdagger(S)\). We also denote by \(\Pf(f)\) the set
\[\Pf(f):=\{P\subset S\st f(P)=P, |P|<\infty\}.\]

Given \(P\in\Pf(f)\), \([f]_P\) denotes the mapping class \([f|_{S\ssm P}]\in\mcg(S\ssm P)\) and \(\tl{[f]_{P}}\) its stable translation length on \(\NC(S\ssm P)\).

We denote by \(\distdagger_S\) the distance in \(\NCdagger(S)\) and by \(\dist_{S\ssm P}\) the distance in \(\NC(S\ssm P)\). If the surface \(S\) is clear from the context, we will simply write \(\distdagger\) and \(\dist_P\).

We recall here the local approximation result \cite[Lemma 3.4]{bhw_quasi} (stated here for the nonseparating fine curve graph), which we will use in our proofs:

\begin{lemma}[\cite{bhw_quasi}]\label{lem:local-approximation} 
Let \(\alpha\) and \(\beta\) be distinct transverse nonseparating curves on \(S\). Suppose  that they are in minimal position in \(S\ssm P\), where \(P\subset S\) is a finite set of points disjoint from \(\alpha\) and \(\beta\). Then
\[\distdagger(\alpha,\beta)=\dist_P([\alpha],[\beta]).\]
\end{lemma}

We will also recall here how subsurface projections are defined (see also \cite{mm_geometryII}). Let \(Z \neq S\) be a subsurface of complexity at least \(4\). The \emph{arc and curve graph} \(\AC(Z)\) of \(Z\) has homotopy classes of curves and arcs as vertices (where the homotopy fixes each boundary component setwise), and two vertices are adjacent if there are disjoint representatives of the corresponding classes. The subsurface projection \(\pi_Z\) is the map \(\pi_Z:\C(S)\to\mathcal{P}(\AC(Z))\), defined as follows:
\begin{itemize}
\item if \(a\in \C(S)\) can be realized disjointly from \(Z\), \(\pi_Z(a)=\emptyset\);
\item if \(a\) has essential intersection with \(Z\), let \(\alpha\) be a representative in minimal position with the boundary of \(Z\). Then \(\pi_Z(a)\) is the set of all homotopy classes of (essential) arcs and curves in \(\alpha\cap Z\).
\end{itemize}

If \(Z\) is an annulus, we look instead at the annular cover \(\tilde{Z}\) and its natural compactification \(\hat{Z}\) obtained by adding circles at infinity. Define \(\AC(Z)\) to be the graph whose vertices are arcs connecting the two boundary components of \(\hat{Z}\), modulo homotopy fixing the boundary pointwise, and where two vertices are connected by an edge if there are representative arcs with disjoint interior. We then define the projection \(\pi_Z:\C(S)\to\mathcal{P}(\AC(Z))\) as follows:
\begin{itemize}
\item if \(a\) has zero intersection with the core of \(Z\), then \(\pi_Z(a)=\emptyset\);
\item if \(a\) intersects \(Z\) essentially, let \(\alpha\) be a representative in minimal position with \(\Z\) and define \(\pi_Z(a)\) to be the set of all homotopy classes of arcs in the lift of \(\alpha\) which join the two boundary components of \(\hat{Z}\).
\end{itemize}

Similarly, we can define the subsurface projection of a lamination \(\lambda\): given a subsurface \(Z\) of complexity at least four, we can put \(\lambda\) in minimal position with respect to \(Z\) and, if it intersects \(Z\), define \(\pi_Z(\lambda)\) as the set of all homotopy classes of curves and arcs in \(\lambda\cap Z\), and \(\pi_Z(\lambda)=\emptyset\) otherwise. We proceed similarly if \(Z\) is an annulus.

Note that in every case, if \(\pi_Z(a)\neq\emptyset\) (for a curve or lamination \(a\)), the arcs and curves in the projection are pairwise disjoint, so the diameter of the projection is at most one.

Given two non-empty subsets \(A\) and \(B\) of \(\AC(Z)\), \(\dist_Z(A,B)\) denotes the Hausdorff distance between \(A\) and \(B\) in the arc and curve graph of \(Z\).

Masur and Minsky proved in \cite{mm_geometryII} the following theorem, called \emph{Bounded Geodesic Image Theorem}:

\begin{thm}[Bounded Geodesic Image Theorem]\label{thm:BGIT}
Let \(S\) be a surface of complexity at least three. Then there is a constant \(B\) depending only on \(S\) such that the following holds: for every subsurface \(Z\) which is either an annulus or has complexity at least four, if \(c_0,\dots,c_n\) is a geodesic such that \(\pi_Z(c_i)\neq\emptyset\) for every \(i\), then \(\diam \bigcup_{i=0}^n\pi_Z(c_i)\leq B\).
\end{thm}

We need the following standard corollary, which follows from
Theorem~\ref{thm:BGIT}, the Morse lemma, and the fact that subsurface
projection is Lipschitz along paths where it is defined:
\begin{cor}[Bounded Quasigeodesic Image Theorem]\label{cor:qBGIT}
  Let \(S\) be a surface of complexity at least three. For every $K>0$
  there is a constant \(B\) depending only on \(S\) and $K$ such that
  the following holds: for every subsurface \(Z\) which is either an
  annulus or has complexity at least four, if \(c_0,\dots,c_n\) is a
  $K$--quasigeodesic path such that \(\pi_Z(c_i)\neq\emptyset\) for
  every \(i\), then \(\diam \bigcup_{i=0}^n\pi_Z(c_i)\leq B\).
\end{cor}
We warn the reader that $c_i$ being a quasigeodesic \emph{path}
(i.e. an edge-path in the curve graph which is a quasigeodesic) is
crucial here.

A subsurface \(Z\) such that \(\pi_Z(a)\neq\emptyset\) for every \(a\in\NC(S)\) is a \emph{witness} for \(N\C(S)\) (see \cite{ms_geometry}, where such surfaces are called holes). Note that witnesses for the nonseparating curve graph are exactly the subsurfaces of full genus. 

\subsection*{Cell-like sets and maps}

We will need a few definitions and a result from point-set topology. A subset \(K\) of a surface is called \emph{cell-like} if it is the intersection of a nested family of closed disks \(D_i\), such that for every \(i\), \(D_{i+1}\subset \text{int}(D_i)\).
The following lemma is well-known. Since we were not able to find a reference, we include a sketch of proof.
\begin{lemma}\label{lem:cell-like}
A subset $C$ of the plane is cell-like if and only if it is compact, connected, and its complement is connected.
\end{lemma}
\begin{proof}
A decreasing intersection of compact and connected spaces is compact and connected.
Furthermore, the complement of a closed disk in the plane is connected, and an increasing union of connected space is connected. The direct implication follows.
Let us assume $C$ is a compact and connected subset of the plane whose complement is connected. By standard analysis arguments, we can find a continuous function $\Phi: \mathbb{R}^2 \to [0, +\infty)$ that vanishes exactly on $C$ and is smooth outside $C$. Let $(\varepsilon_i)$ be a decreasing sequence of regular values of $\Phi$ that converges to $0$ (provided by Sard's theorem). For each $i$, let $D_i$ be the union of the set $\Phi^{-1}([0, \varepsilon_i])$ with all the connected components of its complement that are bounded. Each $D_i$ is a (smooth) surface with boundary, and with connected complement in the plane, thus a disk. We claim that the intersection of the $D_i$'s equals $C$. To prove the non trivial inclusion, consider a point $x$ which is not in $C$. By hypothesis there is a path $\gamma$ from $x$ to infinity. For $i$ large enough, $\epsilon_i$ is less than the infimum of $\Phi$ on $\gamma$, thus $\gamma$ is disjoint from the boundary of $D_i$, and since $D_i$ is a disc, $\gamma$ is also disjoint from $D_i$. In particular $x$ is not in $D_i$.
\end{proof}

Let $S$ be a closed orientable surface of genus at least one, and $p: \widetilde S \to S$ the universal cover. A compact subset $C \subset S$ is said to be \emph{inessential} if it is the image under $p$ of a compact connected set $\widetilde{C}$ which is disjoint from all its translate under automorphisms of $p$. The compact inessential set $C$ is said to be \emph{filled} if, furthermore, the complement of $\widetilde C$ is connected. By the previous lemma, $\widetilde C$ is cell-like, and then it is easy to check that the compact inessential filled set $C$ is also cell-like.

 A continuous map \(q:X\to Y\) between surfaces is called \emph{cell-like} if for every \(y\in Y\), \(q^{-1}(y)\) is cell-like.
We will use the following fact, proven in \cite{mr_topology}:
\begin{thm}\label{thm:cell-like}
Let \(q:X\to Y\) be a cell-like map between surfaces. Then \(q\) is a limit of homeomorphisms (uniformly on compact subsets).
\end{thm}

\section{Proving the comparison results}
This section is dedicated to proving the comparison results: Theorems \ref{thm:dense} and \ref{thm:semi-conjugacy}, and Corollary \ref{cor:torus-case}. The first observation is that in Theorem \ref{thm:dense} we only need to show one inequality, because the other one is clear:

\begin{rmk}\label{rmk:inequality}
For every \(P\in\Pf(f)\), \(\tl{f}\geq \tl{[f]_{P}}\), since the projection \(\NCdagger(S)\to\NC(S\ssm P)\) is distance non-increasing. Hence 
\[\tl{f}\geq \sup_{P\in\Pf(S)}\tl{[f]_{P}}.\]
\end{rmk}

With this remark at hand, let us prove Theorem \ref{thm:dense}.

\begin{proof}[Proof of Theorem \ref{thm:dense}]
By Remark \ref{rmk:inequality}, it is enough to show that \[\tl{f}\leq \sup_{P\in\Pf(S)}\tl{[f]_{P}};\]
we will show that for every \(\varepsilon>0\) there is \(P\in\Pf(f)\) such that \(\tl{[f]_{P}}\geq \tl{f}-\varepsilon\).

If \(\tl{f}=0\), the result is obvious. So suppose \(\tl{f}>0\). Up to taking a power, we can assume that \(\tl{f}\geq L(\delta)\), where \(\delta\) is the hyperbolicity constant for \(\NCdagger(S)\) and \(L(\delta)\) is given by Lemma \ref{lem:quasigeods}.

Choose a curve \(\alpha\) minimizing \(\distdagger(\alpha,f(\alpha))\). Then by Lemma \ref{lem:quasigeods}, there is \(K=K(\delta)\) so that \((f^n(\alpha))_{n\in \Z}\) is a \(K\)-quasigeodesic. Let \(N=N(K+1)\) and \(K'=K'(K+1)\) be given by Lemma \ref{lem:local-to-global}, and \(M\) the Morse constant for \(K'\)-quasigeodesics.

Fix \(\varepsilon>0\); choose \(k\geq N\) such that
\begin{itemize}
\item \(\frac{2M+2}{k}<\frac{\varepsilon}{2}\), and
\item \(\frac{\distdagger(\alpha,f^k(\alpha))}{k}\geq \tl{f}-\frac{\varepsilon}{2}\).
\end{itemize}

For every \(i=1,\dots, k\), let \(\alpha_i\) be a curve transverse to \(\alpha\) and \(\mathcal{C}^0\)-close to \(f^i(\alpha)\), so that \[\distdagger(\alpha_i,f^i(\alpha))\leq 2.\] 
Fix then \(P\in\Pf(f)\) such that \(\alpha\) and \(\alpha_i\), for \(i=1,\dots, k\), are in minimal position with respect to \(P\). The existence of $P$ follows from the hypothesis that the periodic points of $f$ are dense in $S$.
Note that for every \(i\), \(\dist_P([f^i(\alpha)],[\alpha_i])\leq 2\).

We claim that \(([f^n(\alpha)])_n\subset \NC(S\ssm P)\) is a \(k\)-local \((K+4)\)-quasigeodesic, and hence by Lemma \ref{lem:local-to-global} it is a \(K'\)-quasigeodesic. Indeed, suppose \(n,m\in \Z\) satisfy \(|n-m|\leq k\). Without loss of generality, suppose \(n\leq m\). Then, using the assumptions on \(\alpha\), the \(\alpha_i\)'s and Lemma \ref{lem:local-approximation} we get:
\begin{align*}
\dist_P([f^n(\alpha)],[f^m(\alpha)])& =\dist_P([\alpha],[f^{m-n}(\alpha)])\geq\dist_P([\alpha],[\alpha_{m-n}])-2\\
& =\distdagger(\alpha,\alpha_{m-n})-2\geq \distdagger(\alpha,f^{n-m}(\alpha))-4\geq \frac{1}{K}(m-n)-K-4
\end{align*}
and
\[\dist_P([f^n(\alpha)],[f^m(\alpha)])\leq\distdagger(f^n(\alpha),f^{m-n}(\alpha))\leq K(m-n)+K.\]

By Lemma \ref{lem:quasigeods}, for \(M=M(K',\delta)\),
\[\tl{[f]_{P}}\geq \frac{\dist([\alpha],[f^k(\alpha)])}{k}-\frac{2M}{k}\geq \frac{\distdagger(\alpha,f^k(\alpha))}{k}-\frac{2M+2}{k}\geq \tl{f}-\varepsilon.\]
\end{proof}

To prove Theorem \ref{thm:semi-conjugacy} we will need the following:

\begin{lemma}\label{lem:limit-homeos}
Let \(S\) and \(\bar{S}\) be surfaces. Suppose \(q:S\to \bar{S}\) is the limit, for the uniform convergence, of a sequence of homeomorphisms \(q_n\). Let \(\bar{\alpha},\bar{\beta}\in\NCdagger(\bar{S})\) and \(\alpha,\beta\in \NCdagger(S)\) such that \(\alpha\cap q^{-1}(\bar{\alpha})=\emptyset\) and \( \beta\cap q^{-1}(\bar{\beta})=\emptyset\). Then
\[|\distdagger_{\bar{S}}(\bar{\alpha},\bar{\beta})-\distdagger_S(\alpha,\beta)|\leq 2.\]
\end{lemma}

\begin{proof}
Choose sufficiently small regular neighborhoods \(N(q(\alpha))\) of \(q(\alpha)\) and \(N(q(\beta))\) of \(q(\beta)\) such that
\[N(q(\alpha))\cap\bar{\alpha}=\emptyset\]
and
\[N(q(\beta))\cap\bar{\beta}=\emptyset.\]
Then for every \(n\) sufficiently large, \(q_n(\alpha)\subset N(q(\alpha))\) and \(q_n(\beta)\subset N(q(\alpha))\), so \(q_n(\alpha)\cap \bar{\alpha}=\emptyset\) and \(q_n(\beta)\cap \bar{\beta}=\emptyset\). This implies that \(|\distdagger_{\bar{S}}(q_n(\alpha),q_n(\beta))-\distdagger_{\bar{S}}(\bar{\alpha},\bar{\beta})|\leq 2\), As \(q_n\) is a homeomorphism,
\(\distdagger_S(\alpha,\beta)=\distdagger_{\bar{S}}(q_n(\alpha),q_n(\beta))\), so
\[|\distdagger_S(\alpha,\beta)-\distdagger_{\bar{S}}(\bar{\alpha},\bar{\beta})|\leq 2.\]
\end{proof}

\begin{proof}[Proof of Theorem \ref{thm:semi-conjugacy}]
\((1)\) Choose \(\bar{\alpha}\in\NCdagger(\bar{S})\) and \(\alpha\in\NCdagger(S)\) such that \(\alpha\cap q^{-1}(\bar{\alpha})= \emptyset\). By Theorem \ref{thm:cell-like}, \(q\) is a limit of homeomorphisms and we can apply Lemma \ref{lem:limit-homeos} with \(\bar{\beta}=\bar{f}^k(\bar{\alpha})\) and \(\beta=f^k(\alpha)\) and deduce that for every \(k\)
\[|\distdagger_S(\alpha,f^k(\alpha))-\distdagger_{\bar{S}}(\bar{\alpha}, \bar{f}^k(\bar{\alpha}))|\leq 2.\]
In particular the stable translation lengths of \(f\) and \(\bar{f}\) are the same.

\((2)\) Let $\bar{P}, Q$ be as in the statement. Set \(\sat{Q}:=q^{-1}(q(Q))=q^{-1}(\bar{P})\). For every \(x\in Q\), we can choose an open disk \(D_x\) containing \(q^{-1}(q(x))\) such that the closures of all these disks are pairwise disjoint. Given an essential curve \(\alpha\) on \(S\ssm Q\), we can homotope it to a curve \(\alpha'\) disjoint from all \(D_x\). The map \(\varphi\) associating to \([\alpha]\in\NC(S\ssm Q)\) the homotopy class \([\alpha']\in\NC(S\ssm \sat{Q})\) is an isometry between the two graphs, which is equivariant with respect to the action induced by \(f\).

Next, define an isometry \(\psi:\NC(S\ssm\sat{Q})\to\NC(\bar{S}\ssm P)\) as follows. By Theorem \ref{thm:cell-like}, \(q\) is a limit of homeomorphisms \(q_n\). Note that for every nonseparating curve \(\alpha\) on \(S\ssm\sat{Q}\), \([q_n(\alpha)]\in\NC(\bar{S}\ssm P)\) is eventually constant (since the \(q_n(\alpha)\) are essential curves contained in regular neighborhoods of each other). So we set \(\psi([\alpha])=[q_n(\alpha)]\), for any \(n\) sufficiently large so that the homotopy class has stabilized. It is not hard to check that the map is well defined and is distance-preserving. To show that it is a bijection, it is enough to notice that an inverse can be constructed as follows: given \([\beta]\in\NC(\bar{S}\ssm \bar{P})\), choose an open regular neighborhood \(N(\beta)\) of \(\beta\) in $\bar{S} \setminus \bar{P}$. Since \(q^{-1}(N(\beta))=\lim_{n\to \infty}q_n^{-1}(N(\beta))\), \(q^{-1}(N(\beta))\) is an open annulus in $S \setminus \sat{Q}$, so it contains a unique essential curve \(\alpha\), up to homotopy. One can show that the map associating to \([\beta]\) the class \([\alpha]\) is an inverse of \(\psi\). Moreover, \(\psi\) is equivariant with respect to the maps induced by \(f\) and \(\bar{f}\): indeed, for any \([\alpha]\in\NC(S\ssm \sat{Q})\), \(\psi([f]_\sat{Q}([\alpha]))=[q_n(f(\alpha))]\), for every \(n\) sufficiently large, and \([\bar{f}]_{\bar{P}}(\psi([\alpha])=[\bar{f}(q_m(\alpha))])\), for every sufficiently large \(m\). For \(n\) sufficiently large, there is a regular neighborhood \(N(q_n(f(\alpha)))\) which contains \(q(f(\alpha))=\bar{f}(q(\alpha))\), and if \(m\) is sufficiently large, \(q_m(\alpha)\) and \(q(\alpha)\) are close enough so that \[N(q_n(f(\alpha)))\supset \bar{f}(q_n(\alpha)).\] So \(q_n(f(\alpha))\) and \(\bar{f}(q_n(\alpha))\) are homotopic, i.e.\ they define the same class.

This implies that \[[\bar{f}]_{\bar{P}}=(\psi\circ\varphi)\circ [f]_Q \circ (\psi\circ\varphi)^{-1},\] where \(\psi\circ\varphi\) is an isometry. In particular, \([\bar{f}]_{\bar{P}}\) and \([f]_Q\) have the same translation length.
\end{proof}

We end the section by proving the approximation result for the stable translation length on \(\Homeo_0(T^2)\).

\begin{proof}[Proof of Corollary \ref{cor:torus-case}]
If \(f\) is not a hyperbolic isometry, the result is clear. If \(f\) is hyperbolic, its rotation set has non-empty interior by \cite{bhmmw_rotation}. So by \cite{gst_fully}, \(f\) is semi-conjugate via a map \(q\) to \(g\in\Homeo_0(T^2)\) with a dense set of periodic points. The map \(q\) is cell-like since, by \cite[Corollary 5.7]{gst_fully}, preimages of points are inessential filled continua, and thus cell-like by Lemma~\ref{lem:cell-like} and the comments that follows the proof. So by Theorem \ref{thm:semi-conjugacy}
\[\tl{f}=\tl{g}\] 
and by Theorem \ref{thm:dense}
\[\tl{g}=\sup_{P\in\Pf(g)}\tl{[g]_{P}}.\]

Note that for every \(x\in T^2\) which is \(g\)-periodic, the preimage \(q^{-1}(x)\) contains a periodic point for \(f\). Indeed, up to taking powers, we can assume that \(x\) is fixed. Its preimage is a cell-like continuum which is invariant under \(f\). By applying the Cartwright-Littlewood theorem \cite{cl_some}, \(q^{-1}(x)\) contains a fixed point of \(f\).

For every \(P\in\Pf(g)\) we can then choose \(Q_P\in\Pf(f)\) contained in \(q^{-1}(P)\) and such that \(q(Q_P)=P\). By Theorem \ref{thm:semi-conjugacy} we know that \(\tl{[f]_{Q_P}}=\tl{[g]_{P}}\), so:
\[\tl{f}\geq \sup_{Q\in\Pf(f)}\tl{[f]_Q}\geq \sup_{P\in\Pf(g)}\tl{[f]_{Q_P}}=\sup_{P\in\Pf(g)}\tl{[g]_P}=\tl{g}=\tl{f}.\]
\end{proof}

\section{Rationality of stable translation lengths on \(\NC(S)\)}\label{sec:rational}
The main goal of this section is to prove Theorem \ref{thm:rational}. We will deduce it from the following:

\begin{thm}\label{thm:rational-surviving}
Let \(S\) be a surface of genus at least one and \(\varphi\in\mcg(S)\) a pseudo-Anosov mapping class. Then there is \(m>0\) such that \(\varphi^m\) acting on \(\Csurv(S)\) has an invariant geodesic.
\end{thm}

The \emph{surviving curve graph} \(\Csurv(S)\) is the subgraph of \(\C(S)\) whose vertices correspond to homotopy classes of \emph{surviving} curves, i.e.\ curves not bounding a punctured disk. Said differently, surviving curves are those which are still essential (``survive'') after filling in all punctures of \(S\). Note that \(\NC(S)\subset \Csurv(S)\), and we have equality exactly when \(S\) is of genus one. Furthermore, the nonseparating and the surviving curve graphs have the same witnesses.

To deduce Theorem \ref{thm:rational} from Theorem \ref{thm:rational-surviving} we need the following results.

\begin{lemma}\label{lem:isometric-embeddings}
Let \(S\) be a surface of genus at least one. Then for any subsurface \(Z\) of full genus, the inclusions \(\NC(Z)\hookrightarrow\NC(S)\) and \(\Csurv(Z)\hookrightarrow\Csurv(S)\) are isometric embeddings.
\end{lemma}

While the result is certainly known to the experts, we add a proof since we couldn't find it in the literature.

\begin{proof}
We will prove the result for \(\NC(Z)\hookrightarrow\NC(S)\). The same proof works for \(\Csurv(Z)\hookrightarrow\Csurv(S)\).

Let \(a\) and \(b\) be in \(\NC(Z)\) and \(c_0=a,c_1,\dots, c_n=b\) a geodesic between them in \(\NC(S)\). We just need to show that we have a path in \(\NC(Z)\) between \(a\) and \(b\) of length \(n\). To this end, note that since \(Z\) is a witness, its complementary components can only be punctured disks, \(D_1,\dots, D_k\). 

Fix a subsurface \(Z'\) contained in the interior of \(Z\) and isotopic to \(Z\). Let \(P\) be the union, over all components \(D_i\), of all but one puncture of \(D_i\). Let \(f:S\cup P\to \operatorname{int}(Z)\) be a homeomorphism which is the identity on \(Z'\). Choose representatives \(\gamma_i\) of the \(c_i\) in pairwise minimal position so that \(\gamma_0,\gamma_n\subset Z'\). Set \(\gamma_i'=f(\gamma_i)\). Then each \(\gamma_i'\) is a nonseparating curve in \(Z\) and \(|\gamma_i'\cap\gamma_{i+1}'|=|\gamma_i\cap\gamma_{i+1}|\leq 1\), so the homotopy classes \(c_i'\) of the \(\gamma_i'\) form a path in \(\NC(Z)\) between \(a\) and \(b\).
\end{proof}

Let us now show how to deduce Theorem \ref{thm:rational} from Theorem \ref{thm:rational-surviving}:

\begin{proof}[Proof of Theorem \ref{thm:rational}]
Let \(\varphi\in \mcg(S)\) be a hyperbolic isometry of \(\NC(S)\). By Nielsen--Thurston classification and work of Masur and Minsky \cite{mm_geometryI}, there is a subsurface \(Z\) of full genus which is fixed by \(\varphi\) and such that the mapping class \(\psi\) induced on \(Z\) by \(\varphi\) is a pseudo-Anosov mapping class of \(Z\). By Theorem \ref{thm:rational-surviving}, there is some \(m>0\) such that \(\psi^m\) has an invariant geodesic \(\Gamma=(c_i)_{i\in\Z}\) in \(\Csurv(Z)\). We want to show that we can modify it to construct a \(\psi^m\)-invariant geodesic in \(\NC(Z)\). 

If \(Z\) has genus one, \(\Csurv(Z)=\NC(Z)\) and we are done. Suppose then that \(Z\) has genus at least two. We can assume that \(m\geq 2\), so that \(\psi^m(c_0)=c_n\), for some \(n\geq 2\). If \(c_0\) is separating, \(c_{-1}\) and \(c_1\) must be contained in a connected component of \(S\ssm c_0\) (since they are disjoint from \(c_0\) and they have to intersect, as their distance is two). As \(c_0\) is surviving, there is a nonseparating curve \(c_0'\) in the connected component of \(S\ssm c_0\) which does not contain \(c_{-1}\) and \(c_1\). Replace \(\psi^{mk}(c_0)\) by \(\psi^{mk}(c_0')\), for every \(k\in\Z\), and note that we still have a \(\psi^m\)-invariant geodesic. Repeat the process for \(c_1,\dots, c_{n-1}\), replacing them and their iterates under powers of \(\psi^m\) whenever they are separating. The resulting geodesic \(\Gamma'\) is a \(\psi^m\)-invariant geodesic in \(\NC(Z)\). By Lemma \ref{lem:isometric-embeddings}, \(\NC(Z)\) isometrically embeds into \(\NC(S)\), so \(\Gamma'\) is a \(\psi^m\)-invariant geodesic in \(\NC(S)\) made of curves in \(Z\). Thus it is a \(\varphi^m\)-invariant geodesic in \(\NC(S)\).
\end{proof}

For the rest of the section our goal is to prove Theorem \ref{thm:rational-surviving}. This will be done in two steps: first, we show the existence of special geodesics with good local finiteness properties. We then use the same strategy as in \cite{bowditch_tight} to deduce, for any pseudo-Anosov mapping class, the existence of a geodesic preserved by a power of the mapping class.

\begin{rmk} In the rest of the section, \(\dist\) will denote the distance in the surviving curve graph and, given a subsurface \(Z\) of \(S\), \(\dist_Z\) will denote the distance in \(\AC(Z)\). Moreover, some sequences of curves in \(\Csurv(S)\) will be thought of as paths in other graphs (e.g.\ \(\C(S)\)). The words geodesic and quasigeodesic will always refer to geodesics and quasigeodesics in \(\Csurv(S)\), and we will make it explicit if we think of geodesics or quasigeodesics in other graphs. The distance in the curve graph will be denoted by $d_\C$.
\end{rmk}

Let \(\lambda\) be a filling lamination, i.e.\ a lamination whose complementary components contain no essential curve. We say that a curve \(c\) is \emph{\((\lambda,C)\)-like} if for every subsurface \(Z \neq S\) such that \(\pi_Z(c)\neq \emptyset\), we have \(\dist
_Z(\pi_Z(c),\pi_Z(\lambda))\leq C\). A geodesic \(\gamma=(c_i)_i\) in \(\Csurv(S)\) is \emph{\((\lambda,C)\)-like} if \(c_i\) is \((\lambda,C)\)-like for every \(i\).

Our goal, achieved in Proposition~\ref{prop:exist-geo} below, is to show that, given a curve \(a_0\) and a pseudo-Anosov mapping class \(\varphi\) with attracting lamination \(\lambda^+\), there is \(C>0\) so that for every \(n\in \N\) there is a \((\lambda^+,C)\)-like geodesic between \(\varphi^{-n}(a_0)\) and \(\varphi^n(a_0)\).

We will use the following facts:

\begin{lemma}\label{lem:geodesics-between-iterates}
Let \(\varphi\) be a pseudo-Anosov mapping class and \(a_0\in\Csurv(S)\). 
There are constants \(A\), \(K\) and $B$, depending on \(\varphi\) and \(a_0\), such that:
\begin{enumerate}
\item  for every \(n\in\Z\) and for every subsurface \(Z\), if $\pi_Z(\varphi^n(a_0)) \neq \emptyset$ then
\[\dist_Z(\pi_Z(\varphi^n(a_0)),\pi_Z(\lambda^+))\leq A;\]
\item  any geodesic between \(\varphi^{-n}(a_0)\) and \(\varphi^n(a_0)\) is a \(K\)-quasigeodesic in the curve graph. In particular, if \((c_i)_i\) is such a geodesic and \(Z\) a subsurface such that \(\pi_Z(c_i)\neq \emptyset\) for every \(i\) between \(i_0\) and \(j_0\),
\[\dist_Z(\pi_Z(c_{i_0}),\pi_Z(c_{j_0}))\leq B.\]
\end{enumerate}
\end{lemma}

\begin{proof}
  \begin{enumerate}
  \item First note that there is a sequence $m_i\to \infty$, so that
    the geodesic representatives of $\varphi^{m_i}(a_0)$ Hausdorff
    converge to a lamination containing $\lambda^+$ (since $\lambda^+$
    is the attracting fixed point of $\varphi$ on the sphere of
    projective laminations). This implies that for any $Z$ there is a
    power $m$ so that
    \[ \dist_Z(\pi_Z(\varphi^m(a_0)), \pi_Z(\lambda^+)) \leq 4. \] 
      Thus it suffices to show that
    \[\dist_Z(\pi_Z(\varphi^n(a_0)),\pi_Z(\varphi^m(a_0)))\leq A\]
    for all $n,m$. This is a well-known fact, we sketch a proof for
    the convenience of the reader. First, extend the
    orbit $\varphi^k(a_0), k\in\mathbb{Z}$ to a $\varphi$--invariant
    path $\rho$. Since the orbit is a quasigeodesic, we may assume
    that $\rho$ is a $K$-quasigeodesic for some $K>0$.

	Let $B$ denote the constant from the Bounded Geodesic Image Theorem (Corollary~\ref{cor:qBGIT}) for $K$-quasigeodesics.
    Choosing $n_0 > 0$ large enough we see that if
    $\dist_Z(\pi_Z(\varphi^n(a_0)),\pi_Z(\varphi^m(a_0)))> A$, then
    there is an index $i$ so that
    $\dist_Z(\pi_Z(\varphi^i(a_0)),\pi_Z(\varphi^{i+n_0}(a_0)))>
    A-2B$. Namely, the Bounded Quasigeodesic Image Theorem
    allows us to discard initial and
    terminal segments outside the $2$--neighbourhood of $\partial Z$ where the path may 	have empty projection in $Z$, at the cost of $B$ each.

    It thus suffices to find a bound in the case where 
    $m=n+n_0$. But this follows immediately from the fact that the
    subsurface projection between $\varphi^n(a_0), \varphi^{n+n_0}(a_0)$ is always
    bounded by the intersection number of those curves, which equals the intersection 	
    number between $\varphi(a_0), \varphi^{n_0}(a_0)$, independent of $Z$.
  \item Suppose that $(c_i)_i$ is a geodesic in the surviving
    curve graph between \(\varphi^{-n}(a_0)\) and
    \(\varphi^n(a_0)\). Note that $(\varphi^j(a_0))_{j=-n}^n$ is
    a $k_0$-quasigeodesic between the same endpoints, where $k_0$
    depends only on $\varphi$ and $a_0$, since it acts loxodromically on the
    fine curve graph. By the Morse lemma, this implies that for each
    $i$ there is a power $j(i)$ so that
    \[ \dist(c_i, \varphi^{j(i)}(a_0)) \leq M. \]
    Using that $c_i$ is a geodesic, we obtain
    \[ |k - l| = \dist(c_k, c_l) \leq
      d(\varphi^{j(k)}(a_0), \varphi^{j(l)}(a_0)) + 2M \leq
      k_0|j(k)-j(l)| + k_0 + 2M. \] Now, $\varphi$ also acts on the
    curve graph as a loxodromic element, and so
    $(\varphi^j(a_0))_{j=-n}^n$ is a $k_1$-quasigeodesic in the
    curve graph. We thus have
    \[ \dist_\C(c_k,c_l) \geq \dist_\C(\varphi^{j(k)}(a_0),
      \varphi^{j(l)}(a_0)) - 2M \geq \frac{1}{k_1}|j(k)-j(l)| - k_1 -
      2M. \] Combining those two via $d(c_k,c_l) \geq d_\C(c_k,c_l)$, we see that $c_i$ is a
    \(\C(S)\)-quasigeodesic with constants depending only on $k_0,k_1,M$.

    The final sentence is now an immediate consequence of the Bounded
    Quasigeodesic Image Theorem (Corollary~\ref{cor:qBGIT}).
  \end{enumerate}
\end{proof}

The strategy to construct our special geodesics is to start from any geodesic between iterates of \(a_0\) and perform successive modifications until we find a path with the required properties. The basic procedure we will need is substituting a geodesic between two curves \(a,b\in\Csurv(S)\) by a geodesic made by curves contained in the subsurface filled by \(a\) and \(b\):

\begin{lemma}\label{lem:geodesic-in-span}
Let \(a,b\in\Csurv(S)\). Then there is a geodesic joining them such that all curves are contained in \(\Span{a,b}\).
\end{lemma}

\begin{proof}
If \(a\) and \(b\) fill a witness \(Z\) the result follows from the fact that \(\Csurv(Z)\) is isometrically embedded in \(\Csurv(S)\) (Lemma \ref{lem:isometric-embeddings}). Otherwise there is a surviving curve in the complement of \(\Span{a,b}\), so \(\dist(a,b)\leq 2\). If the distance is one the result is obvious, so suppose the distance is two.  Since there is a surviving curve in the complement of \(\Span{a,b}\), not all boundary components of \(\Span{a,b}\) can be the boundary of a punctured disk, i.e.\ be not surviving. So we can pick a surviving boundary component \(c\) and the path \(a,c,b\) is a geodesic contained in \(\Csurv(\Span{a,b})\).
\end{proof}

Next we use Lemma \ref{lem:geodesic-in-span} to prove the existence of geodesics behaving well with respect to subsurface projection.

\begin{lemma}\label{lem:replacement-scheme}
Let \(a,b\in \Csurv(S)\). Then there is a geodesic \(c_0=a,c_1,\dots,c_m=b\) so that for every \(0\leq i\leq j\), \(c_i\subset \Span{c_0,c_j}\). In particular, for any subsurface \(Z\) such that \(\pi_Z(c_1)\neq \emptyset\), either \(\pi_Z(a)\neq \emptyset\), or \(\pi_Z(c_j)\neq \emptyset\) for all \(j\geq 1\).
\end{lemma}

\begin{proof}
By Lemma~\ref{lem:geodesic-in-span}, $a$ and $b$ may be joined by a geodesic 
\(b_0,\dots, b_m\) included in \(\Span{a,b}\). Apply again the lemma to replace the geodesic between $b_0$ and $b_{m-1}$ by a geodesic included in \(\Span{b_0,b_{m-1}}\). Go on successively replacing shorter and shorter initial sub-segments.
Call \(c_0=a,c_1,\dots,c_m=b\) the final geodesic. By construction, each $c_j$ is included in $\Span{c_0, c_{j+1}}$, so that 
\[
\Span{c_0, c_1} \subset \Span{c_0, c_2} \subset \cdots \subset \Span{c_0, c_m}
\]
and thus for every \(0<i\leq j\), \(c_i\) is contained in \(\Span{c_0,c_j}\).


Now let \(Z\) be a subsurface such that \(\pi_Z(c_1)\neq \emptyset\), and suppose \(\pi_Z(c_0)=\emptyset\). Since for every \(j\geq 2\) we have that \(c_1\subset \Span{c_0,c_j}\), \(c_j\) has to have nonempty projection to \(Z\). 
\end{proof}

\begin{prop}\label{prop:exist-geo}
Fix \(a_0\in\Csurv(S)\) and a pseudo-Anosov mapping class \(\varphi\) with attracting lamination \(\lambda^+\). Let \(A\) and \(B\) be the constants from Lemma \ref{lem:geodesics-between-iterates}, and set \(C:=A+2B\). Then for every \(n\) there is a \((\lambda^+,C)\)-like geodesic between \(\varphi^{-n}(a_0)\) and \(\varphi^n(a_0)\).
\end{prop}

\begin{proof}
We will start from any geodesic between the endpoints and construct a finite sequence of geodesics, each obtained by modifying the previous one, such that the last one is \((\lambda^+,C)\)-like. To this end, we will repeatedly use Lemma \ref{lem:replacement-scheme}, see Figure~\ref{fig:good-geodesic}.

Start with any geodesic \(c_0^{(0)}=\varphi^{-n}(a_0),c_1^{(0)},\dots,c_m^{(0)}=\varphi^n(a_0)\). For every \(i\) between \(1\) and \(m-1\), we let \[c_0^{(i)}=\varphi^{-n}(a_0),c_1^{(i)},\dots,c_m^{(i)}=\varphi^n(a_0)\] to be a geodesic obtained from \[c_0^{(i-1)}=\varphi^{-n}(a_0),c_1^{(i-1)},\dots,c_m^{(i-1)}=\varphi^n(a_0)\] by replacing \[c_{i-1}^{(i-1)}=,c_1^{(i-1)},\dots,c_m^{(i-1)}\] with a geodesic given by Lemma \ref{lem:replacement-scheme}. Note that, by construction:
\begin{itemize}
\item  for every \(j\leq i\), \(c_j^{(i)}=c_j^{(i-1)}\);
\item for every \(i\), \(c_0^{(i)}=\varphi^{-n}(a_0)\) and \(c_m^{(i)}=\varphi^{n}(a_0)\).
\end{itemize}

\begin{figure}[hb]
\begin{center}
\begin{overpic}[scale=2]{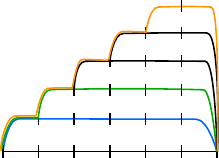}
\put(51,6){\((c_i^{(0)})_i\)}
\put(51,21){\(\textcolor[RGB]{0,100,255}{(c_i^{(1)})_i}\)}
\put(51,35){\(\textcolor[RGB]{0,170,0}{(c_i^{(2)})_i}\)}
\put(45,60){\(\textcolor[RGB]{255,150,0}{(c_i^{(m-1)})_i}\)}
\put(-18,2){\(\varphi^{-n}(a_0)\)}
\put(102,2){\(\varphi^{n}(a_0)\)}
\end{overpic}
\caption{A schematic picture of the construction of a \((\lambda^+,C)\)-like geodesic}\label{fig:good-geodesic}
\end{center}
\end{figure}
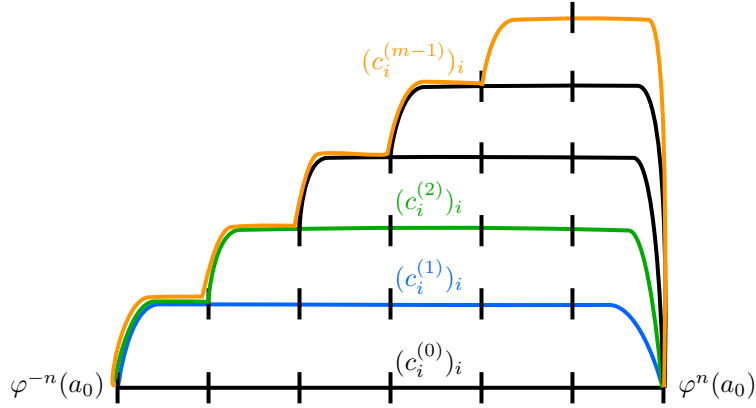

We claim that \((c_i)_i:=(c_i^{(m-1)})_i\) is \((\lambda^+,C)\)-like. Indeed, fix \(i\) and let \(Z\) be a subsurface so that \(\pi_Z(c_i)\neq \emptyset\). If for every \(j\leq i\) the projection of \(c_j\) is not empty, by Lemma \ref{lem:geodesics-between-iterates} we get
\[\dist_Z(\pi_Z(c_i),\pi_Z(\lambda^+))\leq \dist_Z(\pi_Z(c_i),\pi_Z(\varphi^{-n}(a_0)))+\dist_Z(\pi_Z(\varphi^{-n}(a_0)),\pi_Z(\lambda^+))\leq B+A.\]
Otherwise, let \(k\geq 1\) be the smallest integer so that \(\pi_Z(c_{i-k})=\emptyset\). Look at \(c_{i-k+1}\). By Lemma \ref{lem:replacement-scheme} applied to \((c_{j}^{(i-k+1)})_{j=i-k}^m\), we know that \(\pi_Z(c_{j}^{(i-k+1)})\neq \emptyset\) for every \(j\geq i-k+1\), and therefore by Lemma \ref{lem:geodesics-between-iterates}
\[\dist_Z(\pi_Z(c_{i-k+1}),\pi_Z(\varphi^n(a_0)))\leq B.\]
Since \(\pi_Z(c_j)\neq\emptyset\) for every \(j=i-k+1,\dots i\), again by Lemma \ref{lem:geodesics-between-iterates} we have
\[\dist_Z(\pi_Z(c_i),\pi_Z(c_{i-k+1}))\leq B.\]
Combining these two estimates with Lemma \ref{lem:geodesics-between-iterates} we get
\[\dist_Z(\pi_Z(c_i),\pi_Z(\lambda^+))\leq 2B+A.\] 
\end{proof}

Denote by \(\C_{(\lambda^+,C)}(S)\) the set of \((\lambda^+,C)\)-like curves. For \(a,b\in\Csurv(S)\), let \(\G_{(\lambda^+,C)}(a,b)\) be the set of \((\lambda^+,C)\)-like geodesics between \(a\) and \(b\).

We have the following local finiteness result:

\begin{lemma}\label{lem:finiteness}
For every $R$ there is \(N\) depending only on $R$, \(S\) and \(C\) such that for every \(c\in\C_{(\lambda^+,C)}(S)\),
\[\left|B_{R}(c)\cap \C_{(\lambda^+,C)}(S)\right|\leq N.\]
\end{lemma}

To prove this result, we can apply \cite[Theorem A]{watanabe_uniform}, by noting that the diameter of the projection of \(B_{R}(c)\cap \C_{(\lambda^+,C)}(S)\) to any subsurface is bounded by \(\max\{R,2C\}\), and so the size of \(B_{R}(c)\cap \C_{(\lambda^+,C)}(S)\) is bounded by \(N_S(\max\{R,2C\},2)\) from Watanabe's theorem. On the other hand, we can prove the lemma more directly using local finiteness of the marking graph of a surface, as follows.

Key is the following result, which might be well-known.
\begin{lemma}
  There is a constant $D>0$ depending only on the topology of the
  surface so that the following is true.  Suppose $a$ is a curve,
  and $\lambda$ an ending lamination. Then there is a complete clean
  marking $\mu$ (in the sense of Masur-Minsky) so that
  \begin{enumerate}
  \item $a$ is a base curve in $\mu$, and
  \item for every subsurface $Y$ which is disjoint from $\alpha$ we have
    \[ d_Y(\mu, \lambda) \leq D. \]
  \end{enumerate}
\end{lemma}
\begin{proof}
  We first claim: there is a number $D_0$ so that the following is
  true.  If $a$ is any curve and $A$ is a multiarc based on
  $a$ (i.e.\ a finite collection of disjoint properly embedded
  arcs on $S-a$, no two of which are isotopic as such arcs), then
  there is a marking $\mu$ containing $a$ as a base curve which
  intersects each arc in $A$ in at most $D_0$ points. This claim
  follows since there are only finitely many multiarcs up to the
  action of the mapping class group of $S-a$.

  Now the lemma follows by applying the claim for $A$ the set
  of different isotopy classes of arcs of $(S-a)\cap \lambda$. If
  $Y$ is any subsurface disjoint from $\alpha$, we then have that an arc in
  $\pi_Y(\mu)$ and an arc in $\pi_Y(\lambda)$ intersect in at most
  $D_0$ points by construction, showing the lemma for $D = D_0 + 1$.
\end{proof}
We now fix $D$ once and for all, and call $\mu$ with the property of
the lemma a \emph{$\lambda$--extension of $a$}.
\begin{cor}
  There is a constant $K>0$ so that the following is true.  Suppose
  that $a, a'$ are two $(\lambda^+,C)$--like curves, and
  $\mu, \mu'$ are $\lambda^+$--extensions of $a, a'$. Then
  \[ d_{\mathcal{M}}(\mu, \mu') \leq Kd_{\mathcal{C}}(a,a') + K. \]
\end{cor}
\begin{proof}
  This is an immediate consequence of the Masur-Minsky distance
  formula for markings -- every proper subsurface projection term
  \[ d_Y(\mu, \mu') \leq 2\max\{ C, D \} \]
  is uniformly bounded for $Y \neq S$.
\end{proof}

\begin{proof}[Proof of Lemma~\ref{lem:finiteness}]
  Let $\mu$ be a $\lambda^+$--extension of $c$. By the corollary, if
  $d \in B_{R}(c)$ is also $(\lambda^+,C)$-like, and $\eta$ is a
  $\lambda^+$-extension of $d$, then the distance in the marking graph
  between $\mu, \eta$ is at most $R=K(R+1)$. Since the marking
  graph is uniformly locally finite, any $R$--ball in the marking
  graph contains at most $N(R)$ elements. This shows that
  $(3g-3)N(R) = N$ has the desired property.
\end{proof}

A first consequence is the existence of \((\lambda^+,C)\)-like geodesics between the points at infinity of \(\varphi\).

\begin{lemma}\label{lem:between-boundary-points}
There are \((\lambda^+,C)\)-like geodesics between \(\lambda^-\) and \(\lambda^+\).
\end{lemma}

The proof is the same as \cite[Lemma 3.1]{bowditch_tight}, which we sketch for the sake of completeness.

\begin{proof}
Consider \(c_n=\varphi^n(a_0)\), for \(n\in \Z\). For any \(R\geq 0\) and for any sufficiently large \(n,m\), geodesics in \(\G_{(\lambda^+,C)}(c_{-m},c_m)\) and \(\G_{(\lambda^+,C)}(c_{-n},c_n)\) are uniformly close in \(B_R(c_0)\), so by the finiteness result (Lemma \ref{lem:finiteness}) we can find a sequence of geodesics between larger and larger forward and backward iterates of \(a_0\) which agree in \(B_R(c_0)\). Using a diagonal argument, we find a sequence of geodesics agreeing on larger and larger balls, and thus converging to a geodesic between \(\lambda^-\) and \(\lambda^+\). Since being \((\lambda^+,C) \)-like is a local property, the limiting geodesic in \((\lambda^+,C)\)-like.
\end{proof}

Let \(\G_{(\lambda^+,C)}(\lambda^-,\lambda^+)\) be the set of \((\lambda^+,C)\)-like geodesics between \(\lambda^-\) and \(\lambda^+\). Denote by \(G_{(\lambda^+,C)}(\lambda^-,\lambda^+)\) the set of all curves contained in some geodesic in \(\G_{(\lambda^+,C)}(\lambda^-,\lambda^+)\).
\begin{lemma}\label{lem:finite_quotient}
We have \( \varphi(G_{(\lambda^+,C)}(\lambda^-,\lambda^+)) = G_{(\lambda^+,C)}(\lambda^-,\lambda^+) \) and the restriction of \(\varphi\) to \(G_{(\lambda^+,C)}(\lambda^-,\lambda^+)\) has only finitely many orbits.
\end{lemma}

\begin{proof}
The image under $\varphi$ of a \((\lambda^+,C)\)-like curve is a \((\lambda^+,C)\)-like curve, so \(G_{(\lambda^+,C)}(\lambda^-,\lambda^+)\) is invariant under $\varphi$. 
Consider any base point $c_0$, and let \(c_n=\varphi^n(c_0)\), for \(n\in \Z\). The sequence $(c_n)_n$ is a quasigeodesics; let $M$ be the corresponding Morse constant. By the Morse lemma, every point $x \in G_{(\lambda^+,C)}(\lambda^-,\lambda^+)$, which lies on a  \((\lambda^+,C)\)-like geodesics, is within distance $M$ of some $c_n$, so we get
\[
d(x, \varphi(x)) \leq 2M + d(c_n, c_{n+1}) = 2M + d(c_0, \varphi(c_0)) := R.
\]
Now let $\Gamma_0$ be any \((\lambda^+,C)\)-like geodesic, and note that $G_{(\lambda^+,C)}(\lambda^-,\lambda^+)$ is included in the $2\delta$- neighborhood of $\Gamma_0$.
From this we get that every orbit of a point of \(G_{(\lambda^+,C)}(\lambda^-,\lambda^+)\) meets the $2\delta$-neighborhood of a sub-segment of $\Gamma_0$ of length $R + 4\delta$. By local finiteness (Lemma~\ref{lem:finiteness}) applied with $R = 2\delta$, this set has at most $(R + 4\delta)N$ elements.
\end{proof}

Denote $p_\Gamma$ the closest point projection to some \((\lambda^+,C)\)-like geodesic $\Gamma$, and $<_\Gamma$ the order along $\Gamma$ from $\lambda^-$ to $\lambda^+$. The following lemma is certainly well-known, and furthermore it is not really need for the proof of Theorem \ref{thm:rational-surviving}; nevertheless it helps clarifying the proof.
\begin{lemma}\label{lem:order}
Assume $\mathrm{tl}(\varphi) \geq 10 \delta$. 
Then for every $x \in G_{(\lambda^+,C)}(\lambda^-,\lambda^+)$ we have 
\[
p_\Gamma (x) <_\Gamma p_\Gamma (\varphi(x)).
\]
\end{lemma}
\begin{proof}
This inequality is certainly satisfied for some $x$ since the orbit of any $x$ tends to $\lambda^+$. If some $y$ satisfies the reverse inequality, then along any path from $x$ to $y$ in the $2\delta$ neighborhood of $\Gamma$ we will find two adjacent points $x', y'$ such that $x'$ satisfies the wanted inequality while $y'$ satisfies the the reverse one. But since $\varphi(x'), \varphi(y')$ are also adjacent, this implies that $d(x', \varphi (x')) < 10\delta$, which contradicts our hypothesis on the translation length.
\end{proof}

Now, to prove Theorem \ref{thm:rational-surviving} one can use the same argument as in \cite[Lemma 3.4 and Corollary 3.5]{bowditch_tight}. We provide here an alternative argument.

\begin{proof}[Proof of Theorem \ref{thm:rational-surviving}]
Let $\varphi\in\mcg(S)$ be a pseudo-Anosov mapping class.
Lemma~\ref{lem:between-boundary-points} provides a \((\lambda^+,C)\)-like geodesic $\Gamma = (c_{i})_{i\in\Z}$ joining the fixed points $\lambda^-, \lambda^+$ of $\varphi$. Up to replacing $\varphi$ to one of its power, we can assume that $\mathrm{tl}(\varphi) \geq 10 \delta$, and thus Lemma~\ref{lem:order} applies.

Note first that there exists some $m>0$ such that $\Gamma \cap \varphi^{-m}(\Gamma)$ is infinite. Indeed, let $M>0$  be such that none of $\varphi^{-1} \Gamma, \dots, \varphi^{-M} \Gamma$ have infinitely many points in common with $\Gamma$; we look for a contradiction when $M$ is large. Note that these geodesics also pairwise intersect in finitely many points. So for $i$ large enough, no point  of the ball $B_{2\delta}(c_{i})$ belongs to more than one of these geodesics. On the other hand by hyperbolicity, each of these geodesics must meet this ball, and thus the ball must contain at least $M$ points. We get a contradiction as soon as $M$ is larger than the cardinality $N$ of the ball given by the local finiteness (Lemma~\ref{lem:finiteness}).

So fix some $m$ such that $\Gamma \cap \varphi^{-m}(\Gamma)$ is infinite. Since $\Gamma$ meets only finitely many orbits of $\varphi$, the set $\Gamma \cap \varphi^{-m}(\Gamma)$ contains infinitely many points of the orbit of some point $c$. Up to changing $\varphi$ into $\varphi^{-1}$, we may assume that infinitely many such points are included in the positive semi-orbit of $c$. Denote by $\gamma = [c \varphi^m (c)]$ a geodesic segment between \(c\) and \(\varphi^m(c)\), andset  $\ell := \dist(c, \varphi^m (c))$.

\textbf{Claim:} for every $p>0$, we have $\dist(c, \varphi^{pm} (c)) = p\ell$. 

Note that if we prove the claim, the path \[
\bigcup_{p \in \mathbb{Z}} \varphi^{pm} (\gamma)
\]
is a geodesic which is invariant for $\varphi^m$, and we are done.

So let us prove the claim and  let $p>0$. By hypothesis we can find points $c'_{1}, \dots, c'_{p}$ of the positive orbit of $c$ which belong to $\Gamma$, and whose images under $\varphi^m$ also belong to $\Gamma$. Up to extracting (and thanks to lemma~\ref{lem:order}), we can further assume that the order along $\Gamma$ is given by
\[
c,\ c'_{1},\ \varphi^m (c'_{1}),\ c'_{2}, \varphi^m (c'_{2}),\ \dots,\ c'_{p},\ \varphi^m (c'_{p}),
\]
so that there are integers $N_{1}, \dots, N_{p}$ such that these points are obtained from $c$ by applying successively $\varphi^{N_{1}}$, $\varphi^m$, 
$\varphi^{N_{2}}$, $\varphi^m$, and so on. In other words, we have:
\[
c,\ c'_{1}= \varphi^{N_{1}} (c),\ \varphi^m (c'_{1}),\ c'_{2} = \varphi^{m+N_{2}} (c'_{1}),\ \varphi^m (c'_{2}),\ \dots,\ c'_{p} = \varphi^{m+N_{p}} (c'_{p-1}),\ \varphi^m (c'_{p}).
\]
Note that since $\varphi$ is an isometry, for every integer $n$ and every point $x$ of the orbit of $c$, we have
$\dist(x, \varphi^n x) = \dist(c, \varphi^n c)$. For each $i$, let us denote $\ell_{i} = \dist(c, \varphi^{N_{i}} (c))$.
Thus the sub-segment $\Gamma_{0}$ of the geodesic $\Gamma$ from $c$ to $c' = \varphi^{(mp +\sum N_{i})} (c)$ has length $L = p\ell + \sum \ell_{i}$.
We now build another path with the same endpoints by ``cutting and pasting'' to reorder the sub-segments of $\Gamma_0$ by putting those of length $m$ before the all other ones. Namely, we consider the path
\[
\Gamma_{1} = [c \varphi^m (c)] \cup \cdots \cup [\varphi^{(p-1)m} (c) \varphi^{pm} (c)] \cup [\varphi^{pm} (c) \varphi^{pm+N_1} (c)] \cup \cdots \cup
[\varphi^{pm+ N_1 + \cdots + N_{p-1}} (c) \varphi^{pm+ N_1 + \cdots + N_p} (c)]
\]
(where $[xy]$ denotes as usual some geodesic path between $x$ and $y$). Note that $\Gamma_1$ has the same length $L = p\ell + \sum \ell_{i}$ and the same endpoints as $\Gamma_{0}$: thus it is a geodesic path. The claim follows.
\end{proof}

We end this section by showing Proposition \ref{prop:pA-is-rational}, i.e.\ that homeomorphisms which are pseudo-Anosov relatively to a set of points have the same stable translation length as the mapping class relative to the angle-\(\pi\) singularities.

\begin{proof}[Proof of Proposition \ref{prop:pA-is-rational}]
Let \(f\in\Homeo(S)\) be a pseudo-Anosov homeomorphism relatively to a finite set of points.  Fix a stable leaf segment \(s\) not containing any singular point and look at all unstable leaf segments, not passing through singular points, with both endpoints on \(s\) and not intersecting \(s\) in their interior. Collapse \(s\) to a point and note that the collection of loops we obtain generates the fundamental group of the surface, since the unstable foliation is filling. In particular, one such loop needs to be nonseparating.

We can then choose a nonseparating curve \(\alpha\) made by a stable segment \(s\) and an unstable segment \(u\), not passing through any singular point. Let \(p\) and \(q\) be the intersection points of \(s\) and \(u\). A sufficiently small neighborhood \(U\) of \(p\) can be identified with an open subset of \(\R^2\) with the vertical and horizontal foliations, and suppose \(p\) is identified with \((0,0)\). We can assume that \(U\cap s\) corresponds to a segment in the positive \(x\)-axis and and \(U\cap u\) corresponds to a segment in the positive \(y\)-axis. Given \(p'\) in \(U\), we can consider the stable half-leaf starting at \(p'\) and going right from \(p'\) and the unstable half-leaf starting at \(p'\) and going up. If \(p'\) is sufficiently close to \(p\), the half-leaves fellow-travel \(s\) and \(u\) sufficiently close and for sufficiently long so that they meet at a point \(q'\) close to \(q\) and, if we look at the leaf segments \(s'\) and \(u'\) between \(p'\) and \(q'\), we get a curve \(\alpha'\) homotopic to \(\alpha\), disjoint from it, and  not passing through any singular point.

Using this construction, and the density of periodic points for pseudo-Anosov homeomorphisms,
up to moving \(p\) by a small amount and taking a power, we can assume that \(p\) is fixed.

Our goal is to show that for every \(n\geq 0\), \(\distdagger(f^n(\alpha),\alpha)\) differs at most by one to the distance of their classes in \(S\ssm P_\pi\), where \(P_\pi\) is the set of angle-\(\pi\) singularities of \(f\). This implies that the stable translation length of \(f\) acting on \(\NCdagger(S)\) is the same as the stable translation length of \([f]_{P_\pi}\) acting on \(\NC(S\ssm P_\pi)\).

By our assumption, \(f^n(\alpha)=f^n(u)\cup f^n(s)\), where \(f^n(s)\) starts at \(p\) and is contained in \(s\) and \(f^n(u)\) starts at \(p\) and contains \(u\). By moving \(p\) by a small amount, we can find a curve \(\alpha_n\), also made by a stable segment \(s_n\) and an unstable segment \(u_n\), disjoint from \(f^n(\alpha)\) and homotopic to it relative to \(P_\pi\). Moreover, we can suppose that \(s_n\cap u=\emptyset\). Now look at any bigon formed by \(\alpha\) and \(\alpha_n\). This can only be made by a stable leaf segment and an unstable leaf segment.
By Gauss--Bonnet for singular metrics the total curvature in such a
bigon is $2\pi$ or $\pi$, depending on which scenario of
Figure~\ref{fig:parallel-loop} occurs, and thus the bigon has to contain an angle-\(\pi\) singularity. Alternatively, one could use index theory and observe that the total index of one of the foliations in the bigon is $1$ or $1/2$.
Hence, the two curves are in minimal position with respect to \(P_\pi\) and we can use Lemma \ref{lem:local-approximation} to deduce
\[\dist_{P_\pi}([f^n(\alpha)],[\alpha])=\dist_{P_\pi}([\alpha_n],[\alpha])=\distdagger(\alpha_n,\alpha).\]
Since \(|\distdagger(f^n(\alpha),\alpha)-\distdagger(\alpha_n,\alpha)|\leq 1\), we get that
\[|\distdagger(f^n(\alpha),\alpha)-\dist_{P_\pi}([f^n(\alpha)],[\alpha])|\leq 1,\]
as desired.
\end{proof}

\begin{figure}[h]
\begin{center}
\includegraphics{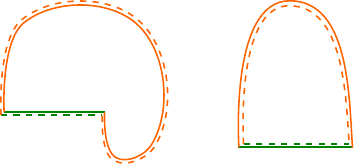}
\caption{Constructing \(\alpha_n\) (dotted) from \(f^n(\alpha)\)}\label{fig:parallel-loop}
\end{center}
\end{figure}

\bibliographystyle{alpha}
\bibliography{references}
\end{document}